\documentclass[12pt]{amsart}
\usepackage{amsmath,amssymb,amsbsy,amsfonts,amsthm,latexsym,
                        amsopn,amstext,amsxtra,euscript,amscd,mathrsfs,color,bm}
\usepackage{float}
\usepackage[english]{babel}
\usepackage{mathtools}
\usepackage[textsize=tiny]{todonotes}
\usepackage{url}
\usepackage[colorlinks,linkcolor=blue,anchorcolor=blue,citecolor=blue]{hyperref}
\usepackage{nicefrac}
\usepackage{enumitem}

\usepackage[a4paper,  left=3.6cm, right=3.6cm, top=3.5 cm, bottom=3.5cm]{geometry}

\usepackage[norefs,nocites]{refcheck}

\begin{document}
%

\newtheorem{theorem}{Theorem}
\newtheorem{lemma}[theorem]{Lemma}
\newtheorem{example}[theorem]{Example}
\newtheorem{algol}{Algorithm}
\newtheorem{corollary}[theorem]{Corollary}
\newtheorem{prop}[theorem]{Proposition}
\newtheorem{proposition}[theorem]{Proposition}
\newtheorem{problem}[theorem]{Problem}
\newtheorem{conj}[theorem]{Conjecture}
\newtheorem{definition}[theorem]{Definition}
\newtheorem{question}[theorem]{Question}
\newtheorem{remark}[theorem]{Remark}
\newtheorem*{acknowledgement}{Acknowledgements}

\newtheorem*{Thm*}{Theorem}
\newtheorem{Thm}{Theorem}[section]
\renewcommand*{\theThm}{\Alph{Thm}}

\numberwithin{equation}{section}
\numberwithin{theorem}{section}
\numberwithin{table}{section}
\numberwithin{figure}{section}

\allowdisplaybreaks

\definecolor{olive}{rgb}{0.3, 0.4, .1}
\definecolor{dgreen}{rgb}{0.,0.5,0.}

\def\cc#1{\textcolor{red}{#1}}

\definecolor{dgreen}{rgb}{0.,0.6,0.}
\def\tgreen#1{\begin{color}{dgreen}{\it{#1}}\end{color}}
\def\tblue#1{\begin{color}{blue}{\it{#1}}\end{color}}
\def\tred#1{\begin{color}{red}#1\end{color}}
\def\tmagenta#1{\begin{color}{magenta}{\it{#1}}\end{color}}
\def\tNavyBlue#1{\begin{color}{NavyBlue}{\it{#1}}\end{color}}
\def\tMaroon#1{\begin{color}{Maroon}{\it{#1}}\end{color}}

%


 \def\mand{\qquad\mbox{and}\qquad}

\def\cA{{\mathcal A}}
\def\cB{{\mathcal B}}
\def\cC{{\mathcal C}}
\def\cD{{\mathcal D}}
\def\cI{{\mathcal I}}
\def\cJ{{\mathcal J}}
\def\cK{{\mathcal K}}
\def\cL{{\mathcal L}}
\def\cM{{\mathcal M}}
\def\cN{{\mathcal N}}
\def\cO{{\mathcal O}}
\def\cP{{\mathcal P}}
\def\cQ{{\mathcal Q}}
\def\cR{{\mathcal R}}
\def\cS{{\mathcal S}}
\def\cT{{\mathcal T}}
\def\cU{{\mathcal U}}
\def\cV{{\mathcal V}}
\def\cW{{\mathcal W}}
\def\cX{{\mathcal X}}
\def\cY{{\mathcal Y}}
\def\cZ{{\mathcal Z}}

\def\C{\mathbb{C}}
\def\F{\mathbb{F}}
\def\K{\mathbb{K}}
\def\Z{\mathbb{Z}}
\def\R{\mathbb{R}}
\def\Q{\mathbb{Q}}
\def\N{\mathbb{N}}
\def\M{\mathrm{M}}
\def\L{\mathbb{L}}
\def\M{{\normalfont\textsf{M}}} 
\def\U{\mathbb{U}}
\def\P{\mathbb{P}}
\def\A{\mathbb{A}}
\def\fp{\mathfrak{p}}
\def\fq{\mathfrak{q}}
\def\n{\mathfrak{n}}
\def\X{\mathcal{X}}
\def\x{\textrm{\bf x}}
\def\w{\textrm{\bf w}}
\def\ovQ{\overline{\Q}}
\def \Kab{\K^{\mathrm{ab}}}
\def \Qab{\Q^{\mathrm{ab}}}
\def \Qtr{\Q^{\mathrm{tr}}}
\def \Kc{\K^{\mathrm{c}}}
\def \Qc{\Q^{\mathrm{c}}}
\def\ZK{\Z_\K}
\def\ZKS{\Z_{\K,\cS}}
\def\ZKSf{\Z_{\K,\cS_{f}}}
\def\RSf{R_{\cS_{f}}}
\def\RTf{R_{\cT_{f}}}

\def\S{\Gcal}
\def\vec#1{\mathbf{#1}}
\def\ov#1{{\overline{#1}}}
\def\sign{{\operatorname{sign}}}
\def\Gm{\G_{\textup{m}}}
\def\fA{{\mathfrak A}}
\def\fB{{\mathfrak B}}

\def \GL{\mathrm{GL}}
\def \Mat{\mathrm{Mat}}

\def\house#1{{%
    \setbox0=\hbox{$#1$}
    \vrule height \dimexpr\ht0+1.4pt width .5pt depth \dp0\relax
    \vrule height \dimexpr\ht0+1.4pt width \dimexpr\wd0+2pt depth \dimexpr-\ht0-1pt\relax
    \llap{$#1$\kern1pt}
    \vrule height \dimexpr\ht0+1.4pt width .5pt depth \dp0\relax}}


\newenvironment{notation}[0]{%
  \begin{list}%
    {}%
    {\setlength{\itemindent}{0pt}
     \setlength{\labelwidth}{1\parindent}
     \setlength{\labelsep}{\parindent}
     \setlength{\leftmargin}{2\parindent}
     \setlength{\itemsep}{0pt}
     }%
   }%
  {\end{list}}

\newenvironment{parts}[0]{%
  \begin{list}{}%
    {\setlength{\itemindent}{0pt}
     \setlength{\labelwidth}{1.5\parindent}
     \setlength{\labelsep}{.5\parindent}
     \setlength{\leftmargin}{2\parindent}
     \setlength{\itemsep}{0pt}
     }%
   }%
  {\end{list}}
\newcommand{\Part}[1]{\item[\upshape#1]}

\def\Case#1#2{%
\smallskip\paragraph{\textbf{\boldmath Case #1: #2.}}\hfil\break\ignorespaces}

\def\Subcase#1#2{%
\smallskip\paragraph{\textit{\boldmath Subcase #1: #2.}}\hfil\break\ignorespaces}

\renewcommand{\a}{\alpha}
\renewcommand{\b}{\beta}
\newcommand{\g}{\gamma}
\renewcommand{\d}{\delta}
\newcommand{\e}{\epsilon}
\newcommand{\f}{\varphi}
\newcommand{\fhat}{\hat\varphi}
\newcommand{\bfphi}{{\boldsymbol{\f}}}
\renewcommand{\l}{\lambda}
\renewcommand{\k}{\kappa}
\newcommand{\lhat}{\hat\lambda}
\newcommand{\bfmu}{{\boldsymbol{\mu}}}
\renewcommand{\o}{\omega}
\renewcommand{\r}{\rho}
\newcommand{\rbar}{{\bar\rho}}
\newcommand{\s}{\sigma}
\newcommand{\sbar}{{\bar\sigma}}
\renewcommand{\t}{\tau}
\newcommand{\z}{\zeta}


\newcommand{\ga}{{\mathfrak{a}}}
\newcommand{\gb}{{\mathfrak{b}}}
\newcommand{\gn}{{\mathfrak{n}}}
\newcommand{\gp}{{\mathfrak{p}}}
\newcommand{\gP}{{\mathfrak{P}}}
\newcommand{\gq}{{\mathfrak{q}}}
\newcommand{\h}{{\mathfrak{h}}}
\newcommand{\Abar}{{\bar A}}
\newcommand{\Ebar}{{\bar E}}
\newcommand{\kbar}{{\bar k}}
\newcommand{\Kbar}{{\bar K}}
\newcommand{\Pbar}{{\bar P}}
\newcommand{\Sbar}{{\bar S}}
\newcommand{\Tbar}{{\bar T}}
\newcommand{\gbar}{{\bar\gamma}}
\newcommand{\lbar}{{\bar\lambda}}
\newcommand{\ybar}{{\bar y}}
\newcommand{\phibar}{{\bar\f}}

\newcommand{\Acal}{{\mathcal A}}
\newcommand{\Bcal}{{\mathcal B}}
\newcommand{\Ccal}{{\mathcal C}}
\newcommand{\Dcal}{{\mathcal D}}
\newcommand{\Ecal}{{\mathcal E}}
\newcommand{\Fcal}{{\mathcal F}}
\newcommand{\Gcal}{{\mathcal G}}
\newcommand{\Hcal}{{\mathcal H}}
\newcommand{\Ical}{{\mathcal I}}
\newcommand{\Jcal}{{\mathcal J}}
\newcommand{\Kcal}{{\mathcal K}}
\newcommand{\Lcal}{{\mathcal L}}
\newcommand{\Mcal}{{\mathcal M}}
\newcommand{\Ncal}{{\mathcal N}}
\newcommand{\Ocal}{{\mathcal O}}
\newcommand{\Pcal}{{\mathcal P}}
\newcommand{\Qcal}{{\mathcal Q}}
\newcommand{\Rcal}{{\mathcal R}}
\newcommand{\Scal}{{\mathcal S}}
\newcommand{\Tcal}{{\mathcal T}}
\newcommand{\Ucal}{{\mathcal U}}
\newcommand{\Vcal}{{\mathcal V}}
\newcommand{\Wcal}{{\mathcal W}}
\newcommand{\Xcal}{{\mathcal X}}
\newcommand{\Ycal}{{\mathcal Y}}
\newcommand{\Zcal}{{\mathcal Z}}

\renewcommand{\AA}{\mathbb{A}}
\newcommand{\BB}{\mathbb{B}}
\newcommand{\CC}{\mathbb{C}}
\newcommand{\FF}{\mathbb{F}}
\newcommand{\G}{\mathbb{G}}
\newcommand{\KK}{\mathbb{K}}
\newcommand{\NN}{\mathbb{N}}
\newcommand{\PP}{\mathbb{P}}
\newcommand{\QQ}{\mathbb{Q}}
\newcommand{\RR}{\mathbb{R}}
\newcommand{\ZZ}{\mathbb{Z}}

\newcommand{\bfa}{{\boldsymbol a}}
\newcommand{\bfb}{{\boldsymbol b}}
\newcommand{\bfc}{{\boldsymbol c}}
\newcommand{\bfd}{{\boldsymbol d}}
\newcommand{\bfe}{{\boldsymbol e}}
\newcommand{\bff}{{\boldsymbol f}}
\newcommand{\bfg}{{\boldsymbol g}}
\newcommand{\bfi}{{\boldsymbol i}}
\newcommand{\bfj}{{\boldsymbol j}}
\newcommand{\bfk}{{\boldsymbol k}}
\newcommand{\bfm}{{\boldsymbol m}}
\newcommand{\bfp}{{\boldsymbol p}}
\newcommand{\bfr}{{\boldsymbol r}}
\newcommand{\bfs}{{\boldsymbol s}}
\newcommand{\bft}{{\boldsymbol t}}
\newcommand{\bfu}{{\boldsymbol u}}
\newcommand{\bfv}{{\boldsymbol v}}
\newcommand{\bfw}{{\boldsymbol w}}
\newcommand{\bfx}{{\boldsymbol x}}
\newcommand{\bfy}{{\boldsymbol y}}
\newcommand{\bfz}{{\boldsymbol z}}
\newcommand{\bfA}{{\boldsymbol A}}
\newcommand{\bfF}{{\boldsymbol F}}
\newcommand{\bfB}{{\boldsymbol B}}
\newcommand{\bfD}{{\boldsymbol D}}
\newcommand{\bfG}{{\boldsymbol G}}
\newcommand{\bfI}{{\boldsymbol I}}
\newcommand{\bfM}{{\boldsymbol M}}
\newcommand{\bfP}{{\boldsymbol P}}
\newcommand{\bfX}{{\boldsymbol X}}
\newcommand{\bfY}{{\boldsymbol Y}}
\newcommand{\bfzero}{{\boldsymbol{0}}}
\newcommand{\bfone}{{\boldsymbol{1}}}

\newcommand{\aff}{{\textup{aff}}}
\newcommand{\Aut}{\operatorname{Aut}}
\newcommand{\Berk}{{\textup{Berk}}}
\newcommand{\Birat}{\operatorname{Birat}}
\newcommand{\characteristic}{\operatorname{char}}
\newcommand{\codim}{\operatorname{codim}}
\newcommand{\Crit}{\operatorname{Crit}}
\newcommand{\critwt}{\operatorname{critwt}} 
\newcommand{\cond}{\operatorname{cond}}
\newcommand{\Cycle}{\operatorname{Cycles}}
\newcommand{\diag}{\operatorname{diag}}
\newcommand{\Disc}{\operatorname{Disc}}
\newcommand{\Div}{\operatorname{Div}}
\newcommand{\Dom}{\operatorname{Dom}}
\newcommand{\End}{\operatorname{End}}
\newcommand{\ExtOrbit}{\mathcal{EO}} 
\newcommand{\Fbar}{{\bar{F}}}
\newcommand{\Fix}{\operatorname{Fix}}
\newcommand{\FOD}{\operatorname{FOD}}
\newcommand{\FOM}{\operatorname{FOM}}
\newcommand{\Gal}{\operatorname{Gal}}
\newcommand{\genus}{\operatorname{genus}}
\newcommand{\GITQuot}{/\!/}
\newcommand{\GR}{\operatorname{\mathcal{G\!R}}}
\newcommand{\Hom}{\operatorname{Hom}}
\newcommand{\Index}{\operatorname{Index}}
\newcommand{\Image}{\operatorname{Image}}
\newcommand{\Isom}{\operatorname{Isom}}
\newcommand{\hhat}{{\hat h}}
\newcommand{\Ker}{{\operatorname{ker}}}
\newcommand{\Ksep}{K^{\textup{sep}}}  
\newcommand{\lcm}{{\operatorname{lcm}}}
\newcommand{\LCM}{{\operatorname{LCM}}}
\newcommand{\Lift}{\operatorname{Lift}}
\newcommand{\limstar}{\lim\nolimits^*}
\newcommand{\limstarn}{\lim_{\hidewidth n\to\infty\hidewidth}{\!}^*{\,}}
\newcommand{\llog}{\log\log}
\newcommand{\logplus}{\log^{\scriptscriptstyle+}}
\newcommand{\maxplus}{\operatornamewithlimits{\textup{max}^{\scriptscriptstyle+}}}
\newcommand{\MOD}[1]{~(\textup{mod}~#1)}
\newcommand{\Mor}{\operatorname{Mor}}
\newcommand{\Moduli}{\mathcal{M}}
\newcommand{\Norm}{{\operatorname{\mathsf{N}}}}
\newcommand{\notdivide}{\nmid}
\newcommand{\normalsubgroup}{\triangleleft}
\newcommand{\NS}{\operatorname{NS}}
\newcommand{\onto}{\twoheadrightarrow}
\newcommand{\ord}{\operatorname{ord}}
\newcommand{\Orbit}{\mathcal{O}}
\newcommand{\Per}{\operatorname{Per}}
\newcommand{\Perp}{\operatorname{Perp}}
\newcommand{\PrePer}{\operatorname{PrePer}}
\newcommand{\PGL}{\operatorname{PGL}}
\newcommand{\Pic}{\operatorname{Pic}}
\newcommand{\Prob}{\operatorname{Prob}}
\newcommand{\Proj}{\operatorname{Proj}}
\newcommand{\Qbar}{{\bar{\QQ}}}
\newcommand{\rank}{\operatorname{rank}}
\newcommand{\Rat}{\operatorname{Rat}}
\newcommand{\Res}{{\operatorname{Res}}}
\newcommand{\Resultant}{\operatorname{Res}}
\renewcommand{\setminus}{\smallsetminus}
\newcommand{\sgn}{\operatorname{sgn}}
\newcommand{\SL}{\operatorname{SL}}
\newcommand{\Span}{\operatorname{Span}}
\newcommand{\Spec}{\operatorname{Spec}}
\renewcommand{\ss}{{\textup{ss}}}
\newcommand{\stab}{{\textup{stab}}}
\newcommand{\Stab}{\operatorname{Stab}}
\newcommand{\Support}{\operatorname{Supp}}
\newcommand{\Sym}{\operatorname{Sym}}  
\newcommand{\tors}{{\textup{tors}}}
\newcommand{\Trace}{\operatorname{Trace}}
\newcommand{\trianglebin}{\mathbin{\triangle}} 
\newcommand{\tr}{{\textup{tr}}} 
\newcommand{\UHP}{{\mathfrak{h}}}    
\newcommand{\Wander}{\operatorname{Wander}}
\newcommand{\<}{\langle}
\renewcommand{\>}{\rangle}

\newcommand{\pmodintext}[1]{~\textup{(mod}~#1\textup{)}}
\newcommand{\ds}{\displaystyle}
\newcommand{\longhookrightarrow}{\lhook\joinrel\longrightarrow}
\newcommand{\longonto}{\relbar\joinrel\twoheadrightarrow}
\newcommand{\SmallMatrix}[1]{%
  \left(\begin{smallmatrix} #1 \end{smallmatrix}\right)}

  \def\({\left(}
\def\){\right)}


\title
{Determining Sidon Polynomials on Sidon Sets over $\F_q\times \F_q$}

\author {Muhammad Afifurrahman}
\author {Aleams Barra}
\address[Muhammad Afifurrahman]{Institut Teknologi Bandung, Jalan Ganesha 10, 40132, Bandung, Indonesia}
\email{afifumuh@gmail.com}
\address[Aleams Barra]{Institut Teknologi Bandung, Jalan Ganesha 10, 40132, Bandung, Indonesia}
\email{aleamsbarra@itb.ac.id}

\thanks{}

\begin{abstract} Let $p$ be a prime, and $q=p^n$ be a prime power. In his works on Sidon sets over $\F_q\times \F_q$, Cilleruelo conjectured about polynomials that could generate $q$-element Sidon sets over $\F_q\times \F_q$.

Here, we derive some criteria for determining polynomials that could generate $q$-element Sidon set over $\F_q\times \F_q$. Using these criteria, we prove that certain classes of monomials and cubic polynomials over $\F_p$ cannot be used to generate $p$-element Sidon set over $\F_p\times \F_p$. We also discover a connection between the needed polynomials and planar polynomials.
\end{abstract}

\maketitle

\tableofcontents
\section{Introduction} Sometimes we use some objects to satisfy some other means, but we do not know whether any other similar object exists.
\subsection{Background}
Let $G$ be an abelian group, written additively. A subset $\Acal \subseteq G$ is a \textit{Sidon set} if, for any $a_1,a_2,a_3,a_4\in \Acal$ that satisfy $a_1-a_2=a_3-a_4$, $\{a_1,a_4\}=\{a_2,a_3\}$.

Sidon sets have been studied extensively since the 1940s, and have appeared on subjects such as finite geometry, graph theory, and coding theory, to mention some instances.

In this paper, we focus on Sidon sets over the group $(\F_q\times \F_q,+)$, with $\F_q$ being a finite field with $q$ being a power of a prime $p$. A precursor of the set first appeared in \cite{erdos} (by considering that the map $\N\times \N \to \Z$ given as $(k,k^2)\to (2pk+k^2)$ with $1\leq k\leq p$ and $\N\times \N \to \F_p\times \F_p$ given as $(k,k^2)\to (k,k^2)$ both gives a Sidon set over the respective groups). The object is resurgent in 2010s, with applications in additive number theory \cite{cil, cilsa} and extremal graph theory \cite{allen, CMT, solymosi, timmons}.

We consider the Sidon sets as being parametrized by two polynomials in $\F_q[x]$. We aim to give some criteria of polynomials that can parametrize a Sidon set, and derive some results regarding structures of a Sidon set over $\F_q\times \F_q$.




\subsection{Maximum Sidon sets}
We now consider a Sidon set $\Acal$ over $(\F_q\times~\F_q,~+)$. We first notice that  when $\F_q$ is of characteristic 2, there do not exist non-trivial Sidon sets over $(\F_q\times \F_q, +)$, since $(a,b)-(c,d)=(c,d)-(a,b)$ for all $(a,b),(c,d)\in\F_{q}\times \F_{q}$. We may now assume that $\characteristic(\F_q)>2$ in this paper.

By a counting argument based on the set $\{a-a', a,a\in \Acal\}$, we get that  $|\Acal|\leq q$. If equality occurs, we say that $\Acal$ is a \textit{maximum} Sidon set over $\F_q\times \F_q$.

Now, let $\F_q=\{x_1,x_2,\dotsc,x_q\}$ be an indexing of $\F_q$. By Lagrange interpolation Theorem, for any $q$-element multiset
\begin{align*}
 X=\{\!\{(a_1,b_1),(a_2,b_2),\dotsc,(a_q,b_q)\}\!\}\subseteq \F_q\times \F_q,
\end{align*} one can always find  $P,Q\in \F_q[x]$ with $\deg P,\deg Q\leq q-1$ such that $(P(x_i),Q(x_i))=(a_i,b_i)$ for $1\leq i \leq q$.  We note that this representation of $X$ is not unique.

By this observation, we see that any maximum Sidon sets over $\F_q\times \F_q$ can be written in the form $$(P,Q):=\{\left(P(x),Q(x)\right)\colon x\in \F_q\}$$ where $P,Q\in \F_q[x]$. Furthermore, we may assume that all polynomials are taken modulo $x^q-x$. Hence, we assume $\deg P, \deg Q \leq q-1$ from now on.

We now restate a family of maximum Sidon sets constructed by Cilleruelo. \begin{lemma}\cite{cil}
\label{lem:OneorTwo}
Let $P,Q\in \F_q[x]$ be non-constant polynomials with degree not more than two, such that for any $k\in \F_q$, $P-kQ$ is not constant. Then, $(P,Q)$ is a maximum Sidon set over $\F_q\times \F_q$. In particular, $(x,x^2)$ is a maximum Sidon set.
\end{lemma}

 It is conjectured in \cite{cand} that when $q=p$, these families are the only possible maximum Sidon sets over $\F_p\times \F_p$. The precise statement is as follows.
\begin{conj}\cite{cand} \label{con:Cand}
Let $p$ be prime, and $A$ be a maximum Sidon set over $\F_p\times \F_p$. Then, there exists $P,Q\in \F_p[x]$ with $1\leq \deg(P),\deg(Q)\leq 2$ and $A=(P,\:Q)$.
\end{conj}

The only progress known to the authors regarding this conjecture is over subsets of the form $(x,\:Q)$ over $\F_p\times \F_p$, which is stated without proof (albeit with typographical errors) in \cite{cand}.

\begin{lemma}\cite{cand}\label{thm:Cand}
If $(x,\:Q)$ is a maximum Sidon set over $\F_p\times \F_p$, then $Q$ is quadratic.
\end{lemma}
We later see that the proof of this statement is immediate from Proposition~\ref{thm:planarSidon}.
On the other hand, for the case $q\neq p$, a maximum Sidon set other than the family in Lemma~\ref{lem:OneorTwo} may exist. This can be seen from Proposition~\ref{thm:planarSidon}. However, it is also not known whether any other maximum Sidon set exists outside of these families.

\section{Sidon polynomials} 

\subsection{Definition of Sidon polynomials}In this paper, we consider a related question to Conjecture~\ref{con:Cand}. First, we define a class of related polynomials as follows:

\begin{definition}
A polynomial $P\in \F_q[x]$ is \textit{Sidon} over $\F_q\times \F_q$ if there exists $Q\in \F_q[x]$ such that $(P,Q)$ is a maximum Sidon set.
\end{definition}

From Cilleruelo's construction in Lemma~\ref{lem:OneorTwo}, and considering that if $(P,Q)$ is a (maximum) Sidon set, $(P,Q+aP)$ is also Sidon, we see that all linear and quadratic polynomials are Sidon polynomials over $\F_q\times \F_q$.

By using this definition, we are able to derive the equivalent form of Conjecture~\ref{con:Cand} in terms of Sidon polynomials. However, in order to do this, we first need an observation and a new definition.

Let $(P,Q)$ be a Sidon set over $\F_q\times \F_q$ and $\sigma :\F_q\rightarrow \F_q$ be a bijection. It is easy to see that $(P\circ \sigma, Q\circ \sigma)$ and $(\sigma\circ P,\sigma \circ Q)$ are Sidon as well.

We note that for any bijective map $\sigma: \F_q\to \F_q$ there exist a polynomial $R\in \F_q[x]$ such that $\sigma(r)=R(r)$ for any $r\in \F_q$. Based on this, we define a polynomial $R\in \F_q[x]$ as a \textit{permutation polynomial} if $R$ is a bijection over $\F_q$. From the observation in the preceding paragraph, we are motivated to define the following equivalence.

\begin{definition} Let $P,P'\in \F_q[x]$.
We say that $P$ and $P'$ are \textit{Sidon equivalent} over $\F_q[x]$, denoted by $P\sim_\cS P'$, if $P'=R\circ P \circ T$, for some permutation polynomials $R,T\in \F_q[x]$. 
When the context is clear, we may only say an equivalence when referring to a Sidon equivalence over $\F_q[x]$.
\end{definition}

 By the definition above and the preceding observation, the following result is immediate.

 \begin{theorem}\label{thm:eqv}
 The relation $\sim_\cS$ is an equivalence relation. Moreover, if $P\sim_\cS~P'$ and $P$ is Sidon, then $P'$ is also Sidon.
 \end{theorem}

We now define two functions over $\F_q[x]$ that act as invariants under relation $\sim_\cS$. Both functions are related to the roots of the polynomial $P(x)-\gamma$ in $\F_q$.

\begin{definition}
For $P\in \F_q[x]$ and nonnegative integer $n$, define $f(P,n)$ as the number of $\gamma\in \F_q$ such that the polynomial $P(x)-\gamma$ has a root  in $\F_q$ with multiplicity at least $n$. Also, define $g(P,n)$ as the number of $\gamma\in \F_q$ such that the equation $P(x)=\gamma$ has exactly $n$ distinct roots in $\F_q$.
\end{definition}

As stated before, these two functions are invariants over $\F_q[x]$ with respect to $\sim_\cS$.  The proof of this statement is given in Section \ref{sec:25}.
\begin{theorem}\label{thm:MS} Let $n$ be a positive integer and $P,P'\in \F_q[x]$ such that $P\sim_\cS P'$. Then, $f(P,n)=f(P',n)$ and $g(P,n)=g(P',n)$
\end{theorem}

As the first application of these invariants, we first classify the equivalencies of linear and quadratic polynomials in $\F_q[x]$ over $\sim_\cS$.

\begin{corollary} Let $a_1,b_1,a_2,b_2,c\in \F_q$ with $a_1,\:a_2\neq 0$ and $\characteristic(\F_q)>2$. Then,
	$$a_1x+b_1\sim_S x$$ and 	$$a_2x^2+b_2x+c\sim_S x^2.$$ However,
$$x \not \sim_\cS x^2.$$
\end{corollary}
\begin{proof} 
We only provide the proof of the last statement. To do this, we  observe that $g(x,2)=0$. However, because the equation $x^2=1$ has two solutions in $\F_q$, we get that $g(x^2,2)>0$. This proves the assertion.
\end{proof}

 Using these notations, we can now restate Conjecture~\ref{con:Cand} in terms of $\sim_\cS$.
 
\begin{conj}
Let $p$ be a prime, and $P\in \F_p[x]$ be a Sidon polynomial. Then, either $P\sim_\cS x$ or $P\sim_\cS x^2$.
\end{conj}
We end this section with noting that the last conjecture may be proven by brute force for $p\leq 5$.

\subsection{Connection with planar polynomials}
A polynomial $P\in \F_q$ is \textit{planar} if, for any nonzero $a\in \F_q$, the polynomial $P(x+a)-P(x)$ is a permutation polynomial. This polynomial family was first introduced by Dembowski and Ostrom in \cite{DO}, with applications in finite geometry.

If $q=p$, it is known that the only planar polynomials are quadratic polynomials \cite{gluck, hiramine, ronyai}. However, if $q$ is not prime, another family of planar polynomial can be constructed \cite{bergman, CL, CM}.

We now state a connection between Sidon polynomials and planar polynomials.
\begin{prop}\label{thm:planarSidon} Any planar polynomial is a Sidon polynomial.
\end{prop}
\begin{proof}

Let $P\in \F_q[x]$ be planar. We now prove that the set $(x,P)$ is Sidon. Now let $x_1,x_2,x_3,x_4$ satisfy the equation \begin{align*}
    (x_1-x_2,P(x_1)-P(x_2))=    (x_3-x_4,P(x_3)-P(x_4)).
\end{align*} Suppose $x_1\neq x_2$ (and $x_3\neq x_4$). By letting $x_1-x_2=x_3-x_4=a$, we see that $P(x_2+a)-P(x_2)=P(x_4+a)-P(x_4)$. However, since $P$ is planar, this implies $x_2=x_4$, and $x_1=x_3$. This proves the initial assertion.
\end{proof}
In fact, by setting $q=p$, this theorem gives a quick proof for Lemma~\ref{thm:Cand}. To end this section, we note that for any $q$, $x$ is a Sidon polynomial over $\F_q \times \F_q$  that is not a planar polynomial in $\F_q$.


\section{Main Results}
After defining Sidon polynomials and its related equivalences, we now ready to present our main results. 

\subsection{Criteria of Sidon polynomials}
We first prove some criteria for determining whether a polynomial in $\F_q[x]$ is Sidon. To do this, we first define some functions over $\F_q[x]$ as follows.

\begin{definition}
For any $P\in \F_q[x]$ and $r\in \F_q$, let
\begin{align*}
d_r(P)&=|\left\{(a,b)\in \F_q\times \F_q \colon P(a)-P(b)=r\}\right|\\
v_{r}(P)&=|\{x\in \F_q\colon P(x)=r\}|.
\end{align*}
 When the discussed polynomial is clear from the context, we simply write $d_r$ and $v_r$ instead of $d_r(P)$ and $v_{r}(P)$.
\end{definition}
 
Now we are ready to state the criteria, as follows:
 \begin{theorem}
\label{thm:Criteria}If $P\in \F_q[x]$ is Sidon, then
\[
d_r(P)\leq\begin{cases}2q-1,& r=0\\ q,& r\neq 0, \end{cases}
\] and 
\[
\sum_{i\in \F_q}v_iv_{i+r}\leq\begin{cases}2q-1,& r=0\\ q,& r\neq 0. \end{cases}
\]
\end{theorem}
 
 \begin{proof}
 	We first bound $d_r(P)$.
 	Let $H=(P,Q)$ be a maximum Sidon set, and $$H -H=\{h_1-h_2,\:h_1,h_2\in H\}.$$ Suppose for a fixed $r\in \F_q$, $(r,w)\in H -H$ for a certain $w\in \F_q$. This implies that the system of equations  \begin{align}\label{eq:SoE}
 P(x)-P(y)&=r \\ \nonumber Q(x)-Q(y)&=w
 	\end{align} has a solution in $\F_q\times \F_q$. Furthermore, since $H$ is Sidon, this solution is unique.
 	
 	If $r\neq 0$, because $H$ is Sidon, each $w\in \F_q$ generates at most one pair $(x,y)\in \F_q\times \F_q$ that satisfy Equations \eqref{eq:SoE}. This implies $$d_r(P)\leq q$$ for all nonzero $r$.
  	
 Now we proceed to bound $d_0(P)$. By the preceding argument, each nonzero $w\in \F_q$ contributes to at most one solution of Equations \eqref{eq:SoE}. Then, when $w=0$, we easily see that the only possible solutions of Equations \eqref{eq:SoE} are when $x=y$. By these observations, we get that $$d_0(P)\leq q-1+q=2q-1,$$ which proves the first inequality.
 	
 	We now prove the second inequality. To do this, we prove that the left-hand sides of both inequalities are equivalent. By  definition of $v_i$, the number of $x,y\in \F_q$ such that $P(x)=i+r$ and $P(y)=i$ is $v_iv_{i+r}$. Summing over all possible $i$, we get \begin{align*}
 		d_r(P)=\sum_{i\in \F_q}v_iv_{i+r},
 	\end{align*} which proves the inequality.
 	
 \end{proof}

For $r=0$, we see that $P=x^2$ satisfy the equality case for both inequalities. For $r\neq 0$, $P=x$ satisfies the equality case.

We note that Theorem~\ref{thm:Criteria} generalizes the observation in Remark 3 of R{\'o}nyai and Sz{\"o}nyi \cite{ronyai}.
 
 Next, we give another criterion to determine a Sidon polynomial. The proof is given in Section \ref{sec:33}.
\begin{theorem}

\label{thm:Criterion}
If $P\in \F_q[x]$ is a Sidon polynomial, then for all $r\in \F_q$,
\begin{align*}
v_{2^{-1}r} + \sum_{i\in \F_q}v_iv_{r-i}\leq 2q.
\end{align*}
\end{theorem}
 
We also notice that when $r=0$, $P=x^2$ and $q\equiv 1\pmod{4}$, equality occurs in Theorem~\ref{thm:Criterion}.

 \subsection{Classifying Sidon polynomials}

 As the first application of the criteria in the previous section, we classify a class of Sidon and non-Sidon monomials. Let $p>3$ be a prime and $q$ be a prime power with $\characteristic(\F_q)>3$.

\begin{corollary}
\label{thm:Monomial} Let $r>2$ be a natural number. Then, \begin{enumerate}[label=(\roman{*})]
\item $P(x)=x^r\in \F_q[x]$ is not Sidon in $\F_q\times \F_q$ if $r\mid q-1$.
\item $P(x)=x^r\in \F_q[x]$ is Sidon in $\F_q\times \F_q$ if $(r,q-1)=1$. \item In $\F_p\times \F_p$, the set $(x^r,Q)$ is a Sidon set if and only if $Q(x^k)$ is a quadratic polynomial (modulo $x^{p}-x$), where $k$ is taken to satisfy $p-1\mid kr-1$.
\end{enumerate}
\end{corollary}
\begin{proof}
	
	We first prove the first statement. Let $P=x^r\in \F_q[x]$ with $r\mid q-1$ and $r>2$. For an arbitrary nonzero $i\in \F_q$ we notice that the equation $x^r=i$ has exactly zero or $r$ solutions in $\F_q$. Furthermore, there are exactly $\dfrac{q-1}{r}$ values of $i$ such that this equation has $r$ solutions. By counting, we may get \begin{align*}
		d_0(P)=|\{(a,b)\in \F_q\times \F_q \colon a^r=b^r\}|=1+r^2\cdot \dfrac{q-1}{r}>2q.
	\end{align*} This violates Theorem~\ref{thm:Criteria}, therefore, this polynomial is not Sidon.
	
	Now we prove the other statements. Firstly, since $(r,q-1)=1$, there exists $k$ such that $q-1\mid kr-1$. Since $(r,q-1)=(k,q-1)=1$, we get that (from \cite{lidl}, for example) $x^r$ and $x^k$ are permutation polynomials over $\F_q$. 
	
	Hence, since composition with permutation polynomial (modulo $x^q-x$) preserve the Sidon property of a set, we see that $(x^r,Q)$ is Sidon if and only if $((x^k)^r,Q(x^k))=(x,Q(x^k))$ is also Sidon. Picking $Q=x^{2r}$, we immediately get that $x^r$ is Sidon. And, from Lemma~\ref{thm:Cand}, we see that $Q(x^k)\pmod{x^q-x}$ must be quadratic, which proves the last statement.
\end{proof}
We then proceed with classifying cubic Sidon polynomials. We first prove in Section \ref{sec:35} that all cubic polynomials in $\F_q[x]$ are classified, with regards to $\sim_\cS$, in one of three equivalence classes. \begin{theorem}\label{thm:cub}

Let $P\in \F_q[x]$ be a cubic polynomial, and $k$ be a nonsquare in $\F_q$. Then, $P$ is equivalent to one of $x^3$, $x^3-x$, and $x^3-kx$. Additionally, the three polynomials are not equivalent with each other.
\end{theorem}

We quickly note that these three polynomials are normalized polynomials.

We also prove that these classes are distinct from the equivalence classes of linear polynomials and quadratic polynomials, except at two special cases. The proof can be seen in Section \ref{sec:36}.

\begin{theorem}\label{thm:cub12}
The classes $x^3$, $x^3-x$, and $x^3-kx$ (with $k$ is a nonsquare) are equivalent to neither $x$ nor $x^2$ in $\F_q[x]$, except in the following cases: \begin{enumerate}
    \item $x^3\sim_\cS x$ when $q \equiv -1\pmod{6}$
    \item $x^3-kx\sim_\cS x^2$ when $q=5$.
\end{enumerate}
\end{theorem}

After classifying all classes of cubic polynomials over $\F_q[x]$, we now apply the criteria of Sidon polynomials that we have to the three polynomials. We quickly see that the following statement follows directly  from Theorem~\ref{thm:Monomial} and Theorem~\ref{thm:cub12}.

\begin{corollary}The polynomial $P(x)=x^3\in \F_q[x]$ is a Sidon polynomial over $\F_q\times \F_q$ if and only if $q \equiv -1\pmod{6}$.
\end{corollary}

For other classes, we have the following classifications when $q=p$ is a prime. The proof is given in Section \ref{sec:Proof}.
\begin{theorem}
\label{thm:NonSidon} Any polynomials that are equivalent to one of these polynomials are not Sidon polynomials in $\F_p\times \F_p$:\begin{enumerate}[label=(\roman{*})]
\item $P(x)=x^3-kx$, if $k$ is a nonsquare in $\F_p$ and $12\mid p+1$.
\item $P(x)=x^3-x$, if $12 \nmid p-1$.
\end{enumerate}
\end{theorem}
The other cases are still open, as per the authors' knowledge. However, the authors believe further criteria and techniques are needed in order to fully solve this problem.
\section{Proof of Theorem~\ref{thm:MS}}\label{sec:25}

\subsection{The invariant $f$}We first prove that $f(P,n)=f(R\circ P\circ T,n)$ for any positive integer $n$ and permutation polynomials $R,T\in \F_q[x]$. For this purpose, we define $$\Fcal_n(Q)=\{\alpha\in \F_q\colon Q(x)-\alpha\text{ has a root of multiplicity at least }n\text{)}\},$$ for a $Q\in \F_q[x]$.  

By definition, we have $|\Fcal_n(Q)|=f(Q,n)$. Hence, in order to prove the original statement, it is sufficient to exhibit a bijection between $\Fcal_n(P)$ and $\Fcal_n(R\circ P\circ T)$. We now prove that the map $\phi(\alpha) = R(\alpha)$ for all $\alpha\in \Fcal_n(Q)$ satisfies this criterion.

First, we observe that since $R$ is a permutation polynomial, the map $\phi$ is a bijection between $\Fcal_n(P)$ and $\phi(\Fcal_n(P))$. It remains to prove that the range of $\phi$ is $\Fcal_n(R\circ P\circ T)$.

We now prove that $R(\alpha)\in \Fcal_n(R\circ P\circ T)$ for each $\alpha\in \Fcal_n(P)$. We first prove that $\alpha\in \Fcal_n(P\circ T)$. Since $\alpha\in \Fcal_n(P)$, there exist $\gamma \in \F_q$ and $Q_1\in \F_q[x]$ with \begin{align*}
P(x)-\alpha=(x-\gamma)^nQ_1(x).
\end{align*}
Substituting $x\to T(x)$, we get \begin{align*}
P(T(x))-\alpha=(T(x)-\gamma)^nQ_1(T(x)).
\end{align*} 
Since $T$ is a permutation polynomial, we see that $T^{-1}(\gamma)$ is a root of $T(x)-\gamma$. Hence, $x-T^{-1}(\gamma)\mid T(x)-\gamma$, and $$(x-T^{-1}(\gamma))^n\mid P(T(x))-\alpha.$$ This implies $T^{-1}(\gamma)$ is a root of $P(T(x))-\alpha$ with order at least $n$, which completes the proof.

We now prove that $R(\alpha)\in \Fcal_n(R\circ P\circ T)$. We observe that $$P(T(x))-\alpha\mid R(P(T(x)))-R(\alpha).$$
Hence, by the relations above, $R(P(T(x)))-R(\alpha)$ has a root of multiplicity at least $n$ as well (namely, $T^{-1}(\gamma)$). This proves the initial assertion, and hence the original statement.

\subsection{The invariant $g$}
We prove that $g(P,n)=g(R\circ P\circ T,n)$ for any positive integer $n$ and permutation polynomials $R,T\in \F_q[x]$. For this purpose, we define  $$\Gcal_n(Q)=\{\alpha \in \F_q \colon Q(x)=\alpha\text{ has exactly }n\text{ distinct  solutions in }\F_q\}$$ for a $Q\in \F_q[x]$.

By definition, we have $|\Gcal_n(Q)|=g(Q,n)$. Hence, in order to prove the original statement, it is sufficient to exhibit a bijection between $\Gcal_n(P)$ and $\Gcal_n(R\circ P\circ T)$.
We now prove that the map $\psi(\alpha) = R(\alpha)$ for all $\alpha\in \Gcal_n(Q)$ satisfies this criterion.

We first observe that since $R$ is a permutation polynomial, the map $\psi$ is a bijection between $\Gcal_n(P)$ and $\psi(\Gcal_n(P))$. It remains to prove that the range of $\psi$ is $\Gcal_n(R\circ P\circ T)$.

We now prove that $R(\alpha)\in \Gcal_n(R\circ P\circ T)$ for each $\alpha\in \Gcal_n(P)$. Because $\alpha\in \Gcal_n(P)$, the equation $P(x)=\alpha$ has exactly $n$ different solutions in $\F_q$. Because $T$ is a permutation polynomial, the equation $$P(T(y))=\alpha$$ also has exactly $n$ solutions in $\F_q$ (by taking $y=T^{-1}(x)$).

Then, because $R$ is a permutation polynomial, the equation $$R(P(T(y)))=R(\alpha)$$ has exactly $n$ solutions in $\F_q$. Hence, $R(\alpha)\in \Gcal_n(R\circ P\circ T)$, which completes the proof of the initial assertion.

\section{Proof of Theorem~\ref{thm:Criterion}}\label{sec:33}

For an additive set $H$, we define $$H+H=\{h+h'\colon h,h'\in H\}.$$ Firstly, we state the following lemma.

\begin{lemma}\label{lem:h+h}
Let $H$ be a Sidon set. If $s\in H+H$, the equation $x+y=s$ has exactly one solution in $H$ if and only if $s=2h$, for an $h\in H$. Otherwise, it has exactly two solutions in $H$.
\end{lemma}

\begin{proof}
Since $s\in H+H$ there exist $(a,b)\in H\times H$ such that $a+b=s$. Suppose $(c,d)\in H\times H$ such that $c+d=s$ as well. Then $a-c=d-b$. Since $H$ is Sidon, then $(a,b)=(c,d)$ or $(a,b)=(d,c)$.  If $a=b$ then $(a,a)$ is the only solution and if $a\neq b$ there are exactly two solutions, namely $(a,b)$ and $(b,a)$.
\end{proof}

Now, let $P\in \F_q[x]$ be a Sidon polynomial and  $H=(P,Q)$ be a Sidon set.
First, we prove that, for a fixed $r\in \F_q$
\begin{align*}
|\{(a,b)\in \F_q\times \F_q \colon P(a)+P(b)=r\}|=\sum_{i\in \F_q}v_iv_{r-i}.
\end{align*} This is done by considering that, for an arbitrary $i\in \F_q$, there are $v_iv_{r-i}$ ways of choosing $a,b\in \F_q$ with $P(a)=i$ and $P(b)=r-i$.

Consider a map $\F_q\times \F_q\to H+ H$, defined as $$(x,y)\mapsto (P(x)+P(y),Q(x)+Q(y)).$$ By Lemma~\ref{lem:h+h}, this map is two-to-one, except when $x=y$.

Now, for a fixed $r\in \F_q$, we consider an element $(r,w)\in H+H$. By definition, there exist $a,b\in \F_q$ that satisfy the system of equations \begin{align*}P(a)+P(b)&=r\\
	Q(a)+Q(b)&=w.
\end{align*}

By the observation in the preceding paragraph, and dividing the cases where $a\neq b$ and $a=b$, we get that there are exactly  \begin{align*}
\dfrac{1}{2} |\{(a,b)\in \F_q\times \F_q, a\neq b \colon P(a)+P(b)=r\}|+v_{2^{-1}r}
\end{align*} elements of the form $(r,w)$ in $H+H$. Since there are at most $q$ distinct element of this form in $H+H$, we get that
\begin{align*} \dfrac{1}{2} |\{(a,b)\in \F_q\times \F_q, a\neq b \colon P(a)+P(b)=r\}|+v_{2^{-1}r}&\leq q.
\end{align*}
On the other hand, \begin{align*}
\sum_{i\in \F_q}v_iv_{r-i}&=|\{(a,b)\in \F_q\times \F_q \colon P(a)+P(b)=r\}|\\
&=|\{(a,b)\in \F_q\times \F_q, a\neq b \colon P(a)+P(b)=r\}|+v_{2^{-1}r}.
\end{align*}
From these two statements, we may get
\begin{align*}
v_{2^{-1}r} + \sum_{i\in \F_q}v_iv_{r-i}\leq 2q,
\end{align*} which proves the theorem.


\section{Proof of Theorem~\ref{thm:cub}}\label{sec:35}
We first prove that for any cubic $P\in \F_q[x]$, there exists a $w\in \F_q$ with $$P\sim_\cS x^3-wx.$$ Let $P=ax^3+bx^2+cx+d$, and $ba^{-1}=e$, $ca^{-1}=f$. We get
\begin{align*}
P&\sim_\cS x^3+ba^{-1}x^2+ca^{-1}x+da^{-1}\\
&\sim_\cS x^3+ex^2+fx \sim_\cS \left(x-\dfrac{e}{3}\right)^3+e\left(x-\dfrac{e}{3}\right)^2+f\left(x-\dfrac{e}{3}\right)\\
&\sim_\cS x^3-\left(f-\dfrac{e^2}{3}\right)x=x^3-wx,
\end{align*} which proves the statement.

Next, we prove that \begin{align*}
	x^3-k_1x \sim_\cS x^3-k_2x,
\end{align*}where $k_1, k_2\in \F_q$ are both nonzero squares in $\F_q$ or are both not squares in $\F_q$. Notice that in either cases, $k_1/k_2$ is a square in $\F_q$. Let $k_1/k_2=r^2$. Then, we have that \begin{align*}
x^3-k_1x=x^3-k_2r^2x\sim_\cS (rx)^3-k_2r^2rx=r^3(x^3-k_2x)\sim_\cS x^3-k_2x,
\end{align*} which proves the statement. 

By this point, we see that any cubic polynomial $P\in \F_q$ is equivalent to either $x^3$, $x^3-x$, or $x^3-kx$, for a fixed nonsquare $k\in \F_q$. Now, we prove that these three polynomials are not equivalent over $\sim_\cS$.

We first prove that $x^3$ is equivalent to neither $x^3-x$ nor $x^3-kx$, with respect to $\sim_S$. By Theorem~\ref{thm:MS}, it suffices to prove $f(x^3,3)\neq f(x^3-x,3)$ and $f(x^3,3)\neq f(x^3-kx,3)$.

We first prove that $$f(x^3-x,3)=f(x^3-kx,3)=0.$$ 

To do this, we first need to proof that there is no $\alpha \in \F_q$ such that $x^3-x-\alpha$ has a triple root. Suppose otherwise; then $x^3-x-\alpha=(x-\gamma)^3$ for a $\gamma$. By equating the $x^2$ and $x$ coefficient, we get that $c=0$. However, the polynomial $x^3-x$ has no triple root. Hence, the assumption is false and  $f(x^3-x,3)=0$. The same reasoning also implies that there are no $\alpha\in \F_q$ such that $x^3-kx-\alpha$ has a triple root, which completes the proof of the equation.

On the other hand, we see that $$f(x^3,3)>0,$$ since the polynomial $x^3$ has $0$ as a triple root. This completes the proof of the initial statement.

Now we prove $x^3-x$ and $x^3-kx$ are not equivalent over $\sim_S$. By Theorem~\ref{thm:MS}, it is sufficient to prove $$f(x^3-x,2)\neq f(x^3-kx,2).$$

Because $f(x^3-x,3)=0$, we see that $f(x^3-x,2)>0$ if and only if there exists an $\alpha\in \F_q$ such that the polynomial $x^3-x-\alpha$ has a double root  $\gamma \in \F_q$.
By differentiating, this is possible if and only if $3\gamma^2-1=0$, which implies $1/3$ is a square in $\F_q$. By the same reasoning, $f(x^3-kx,2)>0$ if and only if $k/3$ is a square. 

Since $k$ is a nonsquare, $1/3$ and $k/3$ cannot both be squares in $\F_q$. Hence, when $f(x^3-x,2)>0$, $f(x^3-kx,2)=0$ and vice versa. This completes the proof.

\section{Proof of Theorem~\ref{thm:cub12}}\label{sec:36}
We first prove that $x^3$ is equivalent to neither $x^2$ nor $x$. By Theorem~\ref{thm:Monomial}, we see that $x^3$ is not Sidon if $6|q-1$, and $x^3$ is a permutation polynomial if $6|q+1$. This completes the proof in this case.

Next, we prove that $x^3-x$ is  equivalent  to neither $x$ nor $x^2$. From Theorem~\ref{thm:MS}, it is sufficient to prove $g(x^3-x,3)$ is equal to neither $g(x,3)$ nor $g(x^2,3)$. We see that $$g(x,3)=g(x^2,3)=0.$$ However, since $x^3-x=0$ has three roots in $\F_q$, we see that $$g(x^3-x,3)>0.$$ This proves the statement.

Now we prove that $x^3-kx$ (where $k$ is a nonsquare) is not equivalent to $x$. Using the normalized permutation polynomial table in \cite[Table~(7.1)]{lidl}, we see that the only  possible normalized permutation polynomials with degree 3 on this case is $x^3$ (if $q\equiv 1\pmod{6})$. However, since we already see that $x^3\not \sim_\cS x^3-kx$, it is not possible for $x^3-kx$ to be a permutation polynomial. Hence, $x^3-kx \not \sim_\cS x$.

Lastly, we prove that $x^3-kx$ is not equivalent to $x^2$ when $q\neq 5$. In order to do this, we compare $g(x^2,2)$ and $g(x^3-kx,2)$. First, since there are exactly $\dfrac{q-1}{2}$ nonzero squares in $\F_q$, we see that $$g(x^2,2)=\dfrac{q-1}{2}.$$

We now calculate $g(x^3-kx,2)$. We notice that the polynomial $x^3-kx-\alpha$ has exactly two roots in $\F_q$ if and only if this polynomial has double root.  Suppose that the root is $\gamma \in \F_q$. Hence, $$3\gamma^2-k=0.$$ This equation has exactly zero or two solutions over $\F_q$. Since each solution corresponds to a different $\alpha$, we see that $$g(x^3-kx,2)\leq 2.$$

Now suppose that $x^2\sim_\cS x^3-kx$. Then, $g(x^2,2)=g(x^3-kx,2)$. Hence, $$\dfrac{q-1}{2}\leq 2.$$ This implies that $q=5$, since $\characteristic(\F_q)>3$.

It remains to prove that  $x^2\sim_\cS x^3-kx$ in $\F_5$. To do this,  we observe that \begin{align*}
(x^3-3x,x^3-2x^2+3x)=\{(0,0),(3,2),(2,1),(3,3),(2,4)\}=(2x^2,x).
\end{align*} We may see that $(2x^2,x)$ is a maximum Sidon set in $\F_5\times \F_5$ from Lemma~\ref{lem:OneorTwo}. Since $3$ is not a square in $\F_5$, this completes the proof of the theorem.

\section{Proof of Theorem~\ref{thm:NonSidon}}\label{sec:Proof}

By Theorem~\ref{thm:Criteria}, we may count the solution of $P(x)=P(y)$ in $\F_p$ to determine whether a polynomial is Sidon over $\F_p\times \F_p$. In the case of $P=x^3-cx$, the equation $P(x)=P(y)$ is equivalent to \begin{align*}
x=y \text{ or } x^2+xy+y^2=c.
\end{align*}

We now calculate the solutions of $x^2+xy+y^2=c$, where $c\in \F_p$ is nonzero.
\begin{prop}\label{lem:aabb}
Let $c\in \F_p$, $c\neq 0$. Then, \begin{align*}
|\{(a,b)\in \F_p\times \F_p \colon a^2+ab+b^2=c\}|=\begin{cases}
p+1,\text{ if } p\equiv -1 {\pmod 6}\\
p-1,\text{ if } p\equiv 1 {\pmod 6}.
\end{cases}
\end{align*}
\end{prop}
\begin{proof}
We apply Theorem 6.26 in \cite{lidl}.
Let $h(a,b)=a^2+ab+b^2$.

We notice that $
    h(a,b)=\begin{pmatrix}a&b\end{pmatrix} \begin{pmatrix}1&2^{-1}\\2^{-1}&1\end{pmatrix}\begin{pmatrix}a\\b\end{pmatrix}$ and $\det\left( \begin{pmatrix}1&2^{-1}\\2^{-1}&1\end{pmatrix} \right) = 2^{-2}\cdot 3.$

By Theorem 6.26 in \cite{lidl}, the number of solutions of $a^2+ab+b^2=c$ over $\F_p\times \F_p$, where $c\neq 0$, is \begin{align*}
    p-\eta(-2^{-2}\cdot 3)=p-\eta(-3),
\end{align*} where $\eta$ is the quadratic character on $\F_p$. This completes the proof.
\end{proof}

We now proceed to count the number of solutions of the equation $P(a)=P(b)$, where $P=x^3-cx\in \F_p[x]$. We recall that this quantity is denoted by $d_0(P)$.

\begin{prop}\label{thm:x3y3}
Let $P(x)=x^3-cx$ with $c\neq 0$. Then, \begin{align*}
d_0(P)=\begin{cases}
2p-3,& p\equiv 1 {\pmod 6},\:c/3\text{ is a square,}\\
2p-1,& p\equiv 1 {\pmod 6},\:c/3\text{ is a nonsquare,}\\
2p-1,& p\equiv -1 {\pmod 6},\:c/3\text{ is a square,}\\
2p+1,& p\equiv -1 {\pmod 6},\:c/3\text{ is a nonsquare.}\\
\end{cases}
\end{align*}
\end{prop}
\begin{proof}By the algebraic manipulation done before, we only need to calculate the number of solutions of $$a^2+ab+b^2=c$$ over $\F_q \times \F_q$, with $a\neq b$.

Without the restriction $a\neq b$, we see that this is just Proposition~\ref{lem:aabb}. Now, we proceed to count the case $a=b$. We see that the equation $$3a^2=c$$ has exactly two solutions if and only if $c/3$ is a square; else, there are no solutions. Adding the $p$ solutions of $P(x)=P(y)$ where $x=y$, the proof can now be completed by careful case division.
\end{proof}

We recall that, from Theorem~\ref{thm:Criteria}, if $P$ is Sidon, $d_0(P)\leq 2p-1$. Hence, by Proposition~\ref{thm:x3y3},  we see that when $p\equiv -1 {\pmod 6}$ and $c/3$ is a nonsquare, $P(x)=x^3-cx$ is not a Sidon polynomial in $\F_p\times \F_p$. Now, notice that $3$ is a square in $\F_p$ if and only if $12\mid p\pm 1$. By using quadratic reciprocity, we get the following results: \begin{enumerate}[label=(\roman{*})]
    \item If $p\equiv -1 {\pmod {12}}$ and $k$ is a nonsquare in $\F_p$, $P(x)=x^3-kx$ is not Sidon.
    \item If $p\equiv 5 {\pmod {12}}$ and $c$ is a square in $\F_p$, $P(x)=x^3-cx$ is not Sidon.
\end{enumerate}These statements are equivalent to the first half of the  Theorem~\ref{thm:NonSidon}.

We now proceed to prove the second half of Theorem~\ref{thm:NonSidon}. Namely, we prove that  $$P(x)=x^3-x$$ is not a Sidon polynomial in $\F_p\times \F_p$ for $p\equiv -1 {\pmod {12}}$ or $p\equiv -5 {\pmod {12}}$. In order to do this, we use the criterion in Theorem~\ref{thm:Criterion}. We first notice that $v_{i}=v_{-i}$. Therefore, by substituting $r=0$ in the left-hand side of Theorem~\ref{thm:Criterion}, we have 
\begin{align*}
v_{2^{-1}\cdot 0} + \sum_{i\in \F_q}v_iv_{0-i}&=v_0+\sum_{i\in \F_q}v_i^2\\
&= v_0 + d_0(P)\\
&=d_0(P)+3,
\end{align*} where the last line can be seen because $x^3-x=0$ has three distinct solutions in $\F_p$.

We first let $p\equiv -1 {\pmod {12}}$. This implies that $p\equiv -1 \pmod{6}$. Since $1/3$ is a square in $\F_p$, we have that  \begin{align*}d_0(P)+3=2p+2
\end{align*} from Proposition~\ref{thm:x3y3}. By Theorem~\ref{thm:Criterion}, $P$ is not Sidon.

Next, let $p\equiv -5 {\pmod {12}}$.  This implies that $p\equiv 1 \pmod{6}$. Since $1/3$ is not a square in $\F_p$, we have that  \begin{align*}d_0(P)+3=2p+2
\end{align*} from Proposition~\ref{thm:x3y3}. Hence, $P$ is not Sidon from Theorem
\ref{thm:Criterion}. This completes the proof of the theorem.

\section*{Acknowledgements}
We would like to thank anonymous reviewers for constructive feedbacks, especially on giving additional references on planar polynomials and simplifying the proof of Lemma~\ref{lem:aabb}. This research is supported by PPMI ITB 2021.

\end{document}